\pdfoutput=1
\documentclass[a4paper]{article}

\usepackage{graphicx}
\usepackage{natbib}
\let\cite=\citep
\setcitestyle{square}

\usepackage{a4wide}
\usepackage{amsmath,amsfonts,amssymb,amsmath,amsthm}
\usepackage{mathrsfs}
\usepackage{graphicx}
\usepackage{subfigure}
\usepackage{ifpdf}

\ifpdf
\usepackage[bookmarks=false,
pdfstartview=FitH,linkbordercolor={0.5 1 1},
citebordercolor={0.5 1 0.5},unicode,
pagebackref,hyperindex]{hyperref}
\else
\fi

\newtheorem{theorem}{Theorem}
\newtheorem{lemma}{Lemma}
\newtheorem{corollary}{Corollary}
\newtheorem{example}{Example}
\newtheorem{remark}{Remark}

\begin{document}

\title{Positive Discrete Spectrum of the Evolutionary Operator of Supercritical
Branching Walks with Heavy Tails}
\author{E. Yarovaya}

%
%
\date{}
\maketitle

\abstract{We consider a continuous-time symmetric supercritical branching
random walk on a multidimensional lattice with a finite set of the particle
generation centres, i.e. branching sources. The main object of study is the
evolutionary operator for the mean number of particles both at an arbitrary
point and on the entire lattice. The existence of positive eigenvalues in the
spectrum of an evolutionary operator results in an exponential growth of the
number of particles in branching random walks, called supercritical in the
such case. For supercritical branching random walks, it is shown that the
amount of positive eigenvalues of the evolutionary operator, counting their
multiplicity, does not exceed the amount of branching sources on the lattice,
while the maximal of these eigenvalues is always simple. We demonstrate that
the appearance of multiple lower eigenvalues in the spectrum of the
evolutionary operator can be caused by a kind of `symmetry' in the spatial
configuration of branching sources. The presented results are based on
Green's function representation of transition probabilities of an underlying
random walk and cover not only the case of the finite variance of jumps but
also a less studied case of infinite variance of jumps.}


\section{Introduction}
\label{sec:intro}

Branching random walks are found to be useful in the investigation of a wide
variety of applications in biology, the theory of homopolymers, population
dynamics, see, e.g. \cite{KimAxel02,CKMV09:e,Clauset11,BMW:MPRF14} and the
bibliography therein.

We consider a symmetric branching random walk (BRW) with continuous time on a
lattice $\mathbb{Z}^{d}$, $d\ge 1$, assuming that, at the initial moment of
time, there is a single particle in the system, which is located at some
point $x$, and birth and death of particles occur at $N$ lattice points
$x_{1}, x_{2},\ldots, x_{N}$, called branching sources.

Continuous-time branching random walks (BRWs) with one particle generation
centre on $\mathbb{Z}^{d}$ and finite variance of jumps of an underlying
random walk have been widely discussed, see, e.g.
\cite{ABY98-1:e,AB00:e,TVY03:e,YarBRW:e,YTV10:e,Y11-SBRW:e} and the
bibliography therein. The presence of a positive eigenvalue in the spectrum
of the corresponding evolution operator ensures an exponential growth of the
number of particles at each lattice point and on the entire lattice. BRWs
with an exponential growth of the number of particles are called
supercritical. By this reason, the authors of the mentioned above
publications usually restricted themselves to determining only the highest
eigenvalue. At the same time, in a number of situations, the information on
whether a positive eigenvalue is unique or nonunique and, in the latter case,
on the location of the other eigenvalues of the evolution operator may be
important for analyzing the behavior of the corresponding BRW.

For example, the uniqueness of a positive eigenvalue substantially
facilitates the study of the propagation of particle fronts
\cite{MolYar12:e,MolYar12-2:e}. However, in the presence of more than one
sources on $\mathbb{Z}^{d}$, the behavior of solutions of differential
equations for the moments of numbers of BRW particles is determined not only
by the value of the leading positive eigenvalue but also by the mutual
arrangement of the positive eigenvalues of the evolution operator
\cite{Y12_MZ:e}. In this connection, we are interested in obtaining
conditions for the emergence of a simple isolated positive eigenvalue in the
spectrum of the evolution operator with increasing the intensity of the
branching sources. We also study the process of the appearance of the
positive eigenvalues with further increasing the intensity of sources. We
show that the appearance of eigenvalues and their multiplicity are determined
not only by the intensities of sources but also by their spatial
configuration. This study makes it possible to reveal the difference in the
behavior of processes on a lattice and on continuous (see, e.g.,
\cite{CKMV09:e}) structures. Such a kind of results were obtained for the
case of finite variance of jumps in \cite{AY15-MPMM:e}. Here we present an
approach allowing to investigate the case of infinite variance of jumps.

The behaviour of the mean number of particles can be described in terms of
the evolutionary operator of a special type \cite{Y12_MZ:e}, which is a
perturbation of the generator $\mathscr{A}$ of a symmetric random walk. In
the case of equal intensities of sources this operator has the form
\begin{equation}\label{E-defY}
\mathscr{H}_{\beta}=\mathscr{A}+
\beta\sum_{i=1}^{N}\delta_{x_{i}}\delta_{x_{i}}^{T},\quad x_{i}\in \mathbb{Z}^{d},
\end{equation}
where $\mathscr{A}:l^{p}(\mathbb{Z}^{d})\to l^{p}(\mathbb{Z}^{d})$,
$p\in[1,\infty]$, is a symmetric operator and $\delta_{x}=\delta_{x}(\cdot)$
denotes a column vector on the lattice taking the value one at the point $x$
and zero otherwise. General analysis of this operator was first done in
\cite{Y12_MZ:e}.

In \cite{Y13-PS:e} it is shown how the operators of type \eqref{E-defY}
appear in BRW models and is demonstrated that the structure of its spectrum
determines the asymptotic behaviour of the numbers of  particles. For the
analysis of the evolutionary equations for the mean number of particles, in
\cite{Y13-PS:e} there was used the technique of differential equations in
Banach spaces. In \cite{Y12_MZ:e} it is shown that $\mathscr{H}_{\beta}$ is a
linear bounded operator in every space $l^{p}(\mathbb{Z}^{d})$,
$p\in[1,\infty]$. All points of its spectrum outside the circle
$C=\{z\in\mathbb{C}: |z-a(0)|\le|a(0)|\}$ with
$a(0)=\delta_{0}^{T}\mathscr{A}\delta_{0}$ may be only eigenvalues of finite
multiplicity. This statement allowed to propose a general method for
obtaining a finite set of equations defining conditions of the existence of
isolated positive eigenvalues in the spectrum of the operator
$\mathscr{H}_{\beta}$ lying outside $C$.

In \cite{Y12_MZ:e} it is shown, that the perturbation of the form $\beta
\sum_{i=1}^{N}\delta_{x_{i}}\delta_{x_{i}}^{T}$ of the operator $\mathscr{A}$
may result in the emergence of positive eigenvalues of the operator
$\mathscr{H}_{\beta}$ and the multiplicity of each of them does not exceed
$N$. However, the number of arising eigenvalues of $\mathscr{H}_{\beta}$ in
\cite{Y12_MZ:e} was not found. In \cite{Y15-DAN:e} it was announced, for the
case of both finite and infinite variance of jumps,
that the maximal eigenvalue of the operator $\mathscr{H}_{\beta}$ is of unit
multiplicity, that is simple, and the total multiplicity of all eigenvalues
does not exceed $N$. This implies, in particular, that in fact the
multiplicity of each eigenvalue of the operator $\mathscr{H}_{\beta}$ does
not exceed $N-1$.
For the case of finite variance of jumps this fact has been proved in
\cite{AY15-MPMM:e}. Below we present a unifying approach for such a kind of
results, including the case of infinite variance of jumps.

The structure of the work is as follows. In Sect.~\ref{sec:model}, a formal
description of BRW with $N$ sources is reminded and the motivation of the
work is explained. In Sect.~\ref{Main Res}, Theorem~\ref{T-new-1} on the
conditions of existence of positive eigenvalues of the evolutionary operator
is obtained. In Sect.~\ref{S-WS-BRW}, we establish the main results on the
structure of the positive discrete spectrum of a supercritical BRW,
Theorem~\ref{T-new-3}, and prove that every supercritical BRW is weakly
supercritical, Theorem~\ref{T-new-2}. Notice, that
Theorems~\ref{T-new-1}--\ref{T-new-2} were announced in \cite{Y15-DAN:e}.
Here we present full proofs for such a kind of results, accentuating the case
of infinite variance of jumps. At last, in Sect.~\ref{S-example} we give an
example, demonstrating the influence of `symmetry' of the sources
configuration on appearance of coinciding eigenvalues in the spectrum of the
operator~\eqref{E-defY}.

\section{Preliminaries}
\label{sec:model}

The evolution of the system of particles in a BRW on $\mathbb{Z}^{d}$ is
defined by the number of particles  $\mu_{t}(y)$ at moment $t$ at each point
$y\in \mathbb{Z}^{d}$ assuming that the system contains only one particle
disposed at some point $x\in \mathbb{Z}^{d}$ at $t=0$, i.e.
$\mu_{0}(y)=\delta_{x}(y)$. Thus, the total number of particles on
$\mathbb{Z}^{d}$ satisfies the equation $\mu_{t}=\sum_{y\in
{\mathbb{Z}}^{d}}\mu_{t}(y)$. The transition probability of the random walk
in the BRW is denoted by $p(t,x,y)$. Let ${\mathsf{E}}_x$ be the expectation
of the total number of particles on the condition that
$\mu_0(\cdot)=\delta_x(\cdot)$. Then, the moments obey $
m_n(t,x,y):={\mathsf{E}}_x \mu^n_t(y)$ and $m_n(t,x):={\mathsf{E}}_x
\mu^n_t$, $n\in {\mathbb{N}}$.

Random walk is specified by a matrix $A=(a(x,y))_{x,y \in\mathbb{Z}^{d}}$ of
transition intensities, where $a(x,y)= a(0,x-y)=a(x-y)$ for all $x$ and $y$.
Thus, the transition intensities are spatially homogeneous and the matrix $A$
is symmetric. The law of walk is described in terms of the function $a(z)$,
$z\in\mathbb{Z}^{d}$, where $a(0)<0$, $a(z)\geq 0$ when $z\neq 0$ and
$a(z)\equiv a(-z)$. We assume that $\sum _{z \in \mathbb{Z}^{d}}a(z)=0$, and
that the matrix $A$ is irreducible, i.e. for all $z\in \mathbb{Z}^{d}$ there
exists a set of vectors $z_{1},z_{2},\ldots,z_{k}\in \mathbb{Z}^{d}$ such
that $z=\sum\limits_{i=1}^{k}z_{i}$ and $a(z_{i})\neq0$ for $i=1,2,\ldots,k$.

We use the function $b_{n}$, $n\ge 0$, where $b_n\ge0$ for $n\ne 1$,
$b_1\le0$, and $\sum_{n} b_n=0$, to describe  branching at a source.
Branching occurs at a finite number of sources, $x_{1},\dots,x_{N}$, and is
given by the infinitesimal generating function $f(u)=\sum_{n=0}^\infty b_n
u^n$ such that $\beta_{r}=f^{(r)}(1)<\infty$ for all $r\in\mathbb{N}$. The
quantity $\beta_{1}=f'(1)$ characterizes the intensity of a source and is
denoted further by $\beta$. The sojourn time of a particle at every source is
distributed exponentially with the parameter $-(a(0)+b_{1})$, see
\cite{Y09-1:e}. The finiteness of all moments is used in the proof of limit
theorems about the behavior of the number of particles by the method of
moments \cite{ShT:e}. In what follows, it suffices to assume only the
existence of $\beta$.

By $p(t,x,y)$ we denote the transition probability of a random walk. This
function is implicitly determined by the transition intensities $a(x,y)$
(see, for example, \cite{GS:e,YarBRW:e}). Then, Green's function of the
operator $\mathscr{A}$ can be represented as a Laplace transform of the
transition probability $p(t,x,y)$:
\[
G_\lambda(x,y):=\int\limits_0^\infty e^{-\lambda t}p(t,x,y)\,dt,\quad \lambda\geq 0.
\]

The analysis of BRWs essentially depends on whether the value of
$G_{0}=G_{0}(0,0)$ is finite or infinite. If the variance of jumps is finite,
that is,
\begin{equation}\label{E:findisp}
\sum\limits_{z\in\mathbb{Z}^{d}} |z|^2 a(z)<\infty,
\end{equation}
where $|z|$ is the Euclidian norm  of the vector $z$, then $G_{0}=\infty$ for
$d=1,2$, and $G_{0}<\infty$ for $d\ge 3$ (see, for example, \cite{YarBRW:e}).
If, for all $z\in\mathbb{Z}^{d}$ with sufficiently large norm, the asymptotic
relation
\begin{equation}\label{E:infindisp}
a(z)\sim\frac{H\left(\frac{z}{|z|}\right)}{|z|^{d+\alpha}},\quad\alpha\in(0,2),
\end{equation}
holds, where $H(\cdot)$ is a continuous positive function symmetric on the
sphere $\mathbb{S}^{d-1}=\{z\in\mathbb{R}^{d} : |z|=1\}$ , then
$G_{0}=\infty$ for $d=1$ and $\alpha \in [1,2)$ and $G_{0}$  is finite if $d
= 1$ and $\alpha \in (0,1)$ or $d\geq 2$ and $\alpha \in (0,2)$
\cite{Y13-CommStat:e}. Condition \eqref{E:infindisp}, unlike
\eqref{E:findisp}, leads to the divergence of the series $a(z)$ and, thereby,
to the infinity of the variance of jumps.


The analysis of the BRW model with one branching source in
\cite{ABY98-1:e,BY-2:e,YarBRW:e,YTV10:e} showed that the asymptotic behaviour
of the mean number of particles at arbitrary point as well as on the entire
lattice is determined by the structure of the spectrum of the linear operator
\eqref{E-defY} when $N=1$. In \eqref{E-defY} the bounded self-adjoint
operator $\mathscr{A}$ in Hilbert space $l^{2}(\mathbb{Z}^{d})$  is a
generator of random walk, and $\beta\varDelta_{x_{1}}$ specifies the
mechanism of branching at the source $x_{1}$. Let us note that the operator
$\mathscr{A}$ is generated by the matrix $A$ of transition intensities. This
model has been generalized in \cite{Y12_MZ:e}, in particular, to the case of
$N$ sources.

The transition probability $p(t,\cdot,y)$ is treated as a function $p(t)$
 in $l^{2}(\mathbb{Z}^{d})$ depending on time $t$ and the
parameter $y$. Then according to \cite{YarBRW:e,Y12_MZ:e} we can rewrite the
evolution equation as the following differential equation in space
$l^{2}(\mathbb{Z}^{d})$:
\[\label{E-p}
\frac{d p}{dt}=\mathscr{A} p, \qquad p(0)=\delta_{y},
\]
where the operator $\mathscr{A}$ acts as
\[
(\mathscr{A}u)(z):=\sum_{z'\in\mathbb{Z}^{d}}a(z-z')u(z').
\]
In the same way we can obtain the differential equation in space
$l^{2}(\mathbb{Z}^{d})$ for the expectation $m_{1}(t,\cdot,y)$ which can be
considered as a function $m_{1}(t)$ in $l^{2}(\mathbb{Z}^{d})$:
\begin{equation}\label{Emrenx}
\frac{d{m}_1}{dt}=\mathscr{H}_{\beta}{m}_{1}, \quad {m}_{1}(0)=\delta_{y}.
\end{equation}
Formally, this equation holds for $ m_{1}(t)={m}_1(t,\cdot)$ on the condition
that ${m}_{1}(0)=1$ in space $l^{\infty}(\mathbb{Z}^{d})$.

It follows from the general theory of linear differential equations in Banach
spaces (see, for example, \cite{DK70:e}) that the investigation of behaviour
of solutions of the equation~\eqref{Emrenx} can be reduced to the analysis of
the spectrum of the linear operators in the right-hand sides of the
corresponding equations. Spectral analysis of the operator
$\mathscr{H}_{\beta}$  of type \eqref{E-defY} was done in \cite{Y12_MZ:e}.

\section{Positive Eigenvalues of the Evolutionary Operator}\label{Main Res}

We denote by $\beta_{c}$ the minimal value of intensity of sources such that
the spectrum of the operator $\mathscr{H}_{\beta}$ contains positive
eigenvalues for $\beta>\beta_{c}$ sufficiently close to $\beta_{c}$. By the
requirement of minimality the quantity $\beta_{c}$ is defined uniquely.

\begin{theorem}\label{T-new-1}
Suppose that a BRW is based on a symmetric spatially homogeneous irreducible
random walk and one of conditions \eqref{E:findisp} or \eqref{E:infindisp}
holds.

If $G_{0}=\infty$, then $\beta_{c}=0$ for $N\geq 1$. If $G_{0}< \infty$, then
$\beta_{c}=G_{0}^{-1}$ for $N=1$ and $0<\beta_{c}<G_{0}^{-1}$ for $N\geq 2$.
\end{theorem}

To prove Theorem~\ref{T-new-1} we will need some auxiliary facts.

\begin{lemma}\label{L1}
The quantity $\lambda>0$ is an eigenvalue of the operator $\mathscr{H}_\beta$
if and only if at least one of the equalities
\begin{equation}\label{E:glbeq1}
\gamma_i(\lambda)\beta = 1, \quad i= 0,\ldots, N-1,
\end{equation}
holds, where $\gamma_i(\lambda)$ are the eigenvalues of the matrix
$\Gamma(\lambda)$ given by the equality
\begin{equation}\label{Def:G}
\Gamma(\lambda)=[\Gamma_{ij}(\lambda)],
\end{equation}
with the elements $\Gamma_{ij}(\lambda)=G_{\lambda}(x_{i},x_{j})>0$, where
$x_{1}, x_{2},\ldots, x_{N}$ are branching sources.
\end{lemma}

The idea of reducing, under appropriate conditions, the infinite-dimensional
eigenvalue problem to a finite-dimensional one is not new, see, e.g.
\cite{Kato:JMSJ57,Kuroda:PJM63,SL:FAP86,AZ:IEOT99}. Lemma~\ref{L1} can be
derived from a more general statement, Theorem~6 from \cite{Y12_MZ:e}, but
for the sake of completeness of exposition we prefer to give below its full
proof.

\begin{proof}
The quantity $\lambda>0$ is an eigenvalue of the operator
$\mathscr{H}_\beta=\mathscr{A} + \beta\sum_{i=1}^N
\delta_{x_{i}}\delta_{x_{i}}^{T}$ if and only if the following equation holds
for some vector $h\neq0$:
\[
\mathscr{A}h + \beta\sum_{i=1}^N \delta_{x_{i}}\delta_{x_{i}}^{T} h = \lambda h.
\]
Let $R_\lambda = (\mathscr{A} - \lambda I)^{-1}$ be the resolvent of the
operator $\mathscr{A}$. By applying $R_\lambda$ to both sides of the last
equation we obtain
\[
h + \beta \ \sum_{i=1}^N R_\lambda\delta_{x_{i}}\delta_{x_{i}}^{T} h = 0.
\]
Since $\delta_{x}\delta_{x}^{T} h = \delta_{x}(\delta_{x}, h)$, then
\[
h + \beta \ \sum_{i=1}^N  (\delta_{x_i}, h) R_\lambda \delta_{x_i} = 0.
\]
Let us scalar left-multiply the last equation  by $\delta_{x_k}$:
\[
(\delta_{x_k}, h) + \sum_{i=1}^n \beta (\delta_{x_i}, h) (\delta_{x_k},
R_\lambda \delta_{x_i}) = 0, \,\,\,k=1,\ldots,n.
\]
By denoting $U_k = (\delta_{x_k}, h)$ we then obtain
\begin{equation}
\label{E-linsys}
U_k + \sum_{i=1}^n \beta U_i (\delta_{x_k}, R_\lambda \delta_{x_i}) = 0, \quad k=1,\ldots,n.
\end{equation}
Thus, the initial equation has a nonzero solution $h$ if and only if the
determinant of the matrix of the obtained linear system is equal to zero. Now
we notice that
\begin{equation*}
(\delta_y, R_\lambda\delta_x) = - \frac{1}{(2\pi)^d} \int_{[-\pi,\pi]^d} \frac{e^{i(\theta, y-x)}}{\lambda - \phi(\theta)}d\theta,
\end{equation*}
where $\phi(\theta) = \sum_{z\in\mathbb Z^d} a(z) e^{i(\theta,z)}$ with
$\theta\in [-\pi,\pi]^d$ is the Fourier transform of the function of
transition probabilities $a(z)$. The right-hand side of the equation can be
represented \cite{YarBRW:e} in terms of Green's function:
\begin{equation}\label{E-Green}
G_\lambda(x,y) := \int_0^\infty e^{-\lambda t} p(t,x,y)dt =
\frac{1}{(2\pi)^d} \int_{[-\pi,\pi]^d} \frac{e^{i(\theta, y-x)}}{\lambda -
\phi(\theta)}d\theta.
\end{equation}
Hence $ (\delta_y, R_\lambda\delta_x)=-G_\lambda(x,y). $ It implies that the
condition of vanishing of the determinant of the linear system
\eqref{E-linsys} can be rewritten as $\det
\left(\beta\Gamma(\lambda)-I\right) = 0,$ which is equivalent to
equation~\eqref{E-det} when $\beta\neq0$.

Notice, that representation \eqref{E-Green} implies the inequalities
$\Gamma_{ij}(\lambda)=G_{\lambda}(x_{i},x_{j})>0$, that is the matrix
$\Gamma(\lambda)$ is positive.

Recalling that the eigenvalues of the matrix $\Gamma(\lambda)$ are denoted by
$\gamma_i(\lambda)$, where $i=0,\ldots,N-1$, we obtain that \eqref{E-det}
holds for some $\beta$ and $\lambda$ if and only if \eqref{E:glbeq1} is true.
The lemma is proved.
\end{proof}

\begin{lemma}\label{L3}
Let $Q=(q_{ij})$ be a matrix with elements
\begin{equation*}
q_{ij}=\frac{1}{(2\pi)^d} \int_{[-\pi,\pi]^d} q(\theta) e^{i(\theta, x_{i}-x_{j})}d\theta,
\end{equation*}
where $x_{1},x_{2},\ldots,x_{N}$ is a set of linearly independent vectors,
and $q(\theta)\geq q_{*}>0$ is an even function summable on $[-\pi,\pi]^d$.
Then $Q$ is a real, symmetric and positive-definite matrix satisfying
$(Qz,z)\geq q_{*}( z,z )$.
\end{lemma}
\begin{proof}
The matrix $Q$ is real and symmetric since the function $q(\theta)$ is even
and then
\begin{equation*}
q_{ij}=\frac{1}{(2\pi)^d} \int_{[-\pi,\pi]^d} q(\theta)\cos(\theta, x_{i}-x_{j}) d\theta.
\end{equation*}
Thus, we need to prove only that the matrix $Q$ is positive-definite. By
definition, $ (Qz,z) = \sum_{i,j=1}^{N} q_{ij} z_i z_j, $ where $z = (z_1,
z_2, \ldots, z_N)$. Then
\begin{align*}
\left( Qz,z \right)&=\frac{1}{(2\pi)^d}
\sum_{i,j=1}^{N} \int_{[-\pi,\pi]^d} q(\theta)e^{i(\theta, x_{i}-x_{j})}z_{i}z_{j}d\theta\\
&=\frac{1}{(2\pi)^d}
\sum_{i=1}^{N}\sum_{j=1}^{N} \int_{[-\pi,\pi]^d} q(\theta)\left(e^{i(\theta, x_{i})}z_{i}\right)
\left(e^{-i(\theta,x_{j})}z_{j}\right)d\theta\\
&=\frac{1}{(2\pi)^d}
\int_{[-\pi,\pi]^d} q(\theta)\sum_{i=1}^{N}\sum_{j=1}^{N} \left(e^{i(\theta, x_{i})}z_{i}\right)
\left(e^{-i(\theta,x_{j})}z_{j}\right)d\theta\\
&=\frac{1}{(2\pi)^d}
\int_{[-\pi,\pi]^d} q(\theta)\left|e^{i(\theta, x_{1})}z_{1}+\cdots+e^{i(\theta, x_{N})}z_{N}\right|^{2}d\theta\ge 0.
\end{align*}
Since $q(\theta)\geq q_{*}$, then
\begin{align*}
\left( Qz,z \right)&\ge
\frac{q_{*}}{(2\pi)^d}
\int_{[-\pi,\pi]^d}\left|e^{i(\theta, x_{1})}z_{1}+\cdots+e^{i(\theta, x_{N})}z_{N}\right|^{2}d\theta\\
&=\frac{q_{*}}{(2\pi)^d}
\int_{[-\pi,\pi]^d}\left(z^{2}_{1}+\cdots+z^{2}_{N}+\sum_{i\neq j}\left(e^{i(\theta, x_{i})}z_{i}\right)
\left(e^{-i(\theta,x_{j})}z_{j}\right)\right)d\theta.
\end{align*}
The integral of the summands $\left(e^{i(\theta, x_{i})}z_{i}\right)
\left(e^{-i(\theta,x_{j})}z_{j}\right)$ vanishes, since $x_i \ne x_j$ for $i
\ne j$. Then the integral in the right-hand side of the equation can be
calculated explicitly and is equal to $q_* (z_1^2 + \ldots + z_N^2)$. So,
$(Qz,z) \geq q_* (z,z)$. The lemma is proved.
\end{proof}

\begin{lemma}\label{L2}
Each of the functions $\gamma_{i}(\lambda)$ is strictly decreasing when
$\lambda\ge0$. Moreover, the total number of solutions $\lambda_{i}(\beta)$
of the equations $\gamma_i(\lambda)\beta = 1$, $i = 0,\ldots, N-1$, does not
exceed $N$, that is the number of eigenvalues of the operator~\eqref{E-defY}
does not exceed $N$.
\end{lemma}

\begin{proof}
Let the eigenvalues $\gamma_{i}(\lambda)$ of the matrix~\eqref{Def:G} be
arranged in decreasing order:
\begin{equation*}
0\leq\gamma_{N-1}(\lambda)\leq \ldots \leq \gamma_{1}(\lambda)\le\gamma_{0}(\lambda).
\end{equation*}
Let us consider the matrix
\begin{eqnarray*}
\Gamma(\lambda_{1},\lambda_{2}):=\Gamma(\lambda_{1})-\Gamma(\lambda_{2}).
\end{eqnarray*}

Taking
\[
q(\theta)=\frac{1}{\lambda-\phi(\theta)}
\]
we obtain for $\lambda>0$ that
\begin{equation*}
q(\theta)\ge
\frac{1}{\lambda+s}>0,
\end{equation*}
where $s=\max_{{\theta}\in [-\pi,\pi]^d} \{-\phi(\theta)\}>0$. Hence, by
Lemma~\ref{L3} for each $\lambda>0$ the matrix $\Gamma(\lambda)$ defined by
\eqref{Def:G} and \eqref{E-Green} is real, symmetric and positive-definite.

From \eqref{Def:G} and \eqref{E-Green} we obtain also that the elements of
the matrix $\Gamma(\lambda_{1},\lambda_{2})$ are as follows
\begin{equation*}
\Gamma_{ij}(\lambda_{1},\lambda_{2})
=(\lambda_{2}-\lambda_{1})
\frac{1}{(2\pi)^d} \int_{[-\pi,\pi]^d} \frac{e^{i(\theta, x_{i}-x_{j})}}{(\lambda_{1}-\phi(\theta))(\lambda_{2}-\phi(\theta))}d\theta.
\end{equation*}
For the continuous function
\begin{equation*}
q(\theta)=\frac{1}{(\lambda_{1}-\phi(\theta))(\lambda_{2}-\phi(\theta))}
\end{equation*}
we have the lower bound
\begin{equation*}
q(\theta)\ge q_{*}(\lambda_{1},\lambda_{2}):=\frac{1}{(\lambda_{1}+s)(\lambda_{2}+s)}>0.
\end{equation*}
Hence, again by Lemma~\ref{L3} the matrix $\Gamma(\lambda_1,\lambda_2)$ is
self-adjoint and positive-definite when $\lambda_2 > \lambda_1$. In this case
the Weyl theorem \cite[Thm.~4.3.1]{HJ:e} implies, for all $i=0,\ldots, N-1$,
the inequalities
\begin{equation*}
\gamma_{i}(\lambda_{1})- \gamma_{i}(\lambda_{2})\geq q_{*}(\lambda_{1},\lambda_{2})>0,
\end{equation*}
since the minimal eigenvalue of the matrix $\Gamma(\lambda_1,\lambda_2)$ has
the lower bound equal to $q_{*}(\lambda_1,\lambda_2)$. So,
$\gamma_i(\lambda_1) > \gamma_i(\lambda_2)$ when $\lambda_2 > \lambda_1$,
that is the function $\gamma_i(\lambda)$ is strictly decreasing with respect
to $\lambda$.

Since the functions $\gamma_i(\lambda)$ are strictly decreasing then each of
the equations $\gamma_{i}(\lambda)\beta=1$, where $i=0,\ldots,N-1$, for each
$\beta$ has no more than one solution (the eigenvalue of the operator
$\mathcal{H}_\beta$). So, the total amount of the eigenvalues of the operator
$\mathcal{H}_\beta$ does not exceed $N$. The lemma is proved.
\end{proof}

Now, we are ready to prove Theorem~\ref{T-new-1}.

\begin{proof}[Proof of Theorem~\ref{T-new-1}]
For the case when condition \eqref{E:findisp} holds the theorem is proved in
\cite{AY15-MPMM:e}. So, we will consider only the case when condition
\eqref{E:infindisp} is fulfilled.

From the integral representation~\eqref{E-Green} of the function
$\phi(\theta)$ it follows that
\begin{equation}\label{E-alphaest}
c_{1}\|\theta\|^{\alpha} \leqslant
|\phi(\theta)| \leqslant c_{2}\|\theta\|^{\alpha}
\end{equation}
for some nonzero real constants $c_{1}$ and $c_{2}$
\cite{Y13-CommStat:e,NPSMDA-Y14,Koz:INFOPROC15:e}. Then, for $\lambda=0$, we
have
\[
G_{0}=
\int_{0}^{\infty}p(t)\,dt=
\frac{1}{(2\pi)^d}\int\limits_{[-\pi,\pi]^d}
\frac{d\theta}{(-\phi(\theta))}\leq
 \frac{1}{(2\pi)^d}\int\limits_{[-\pi,\pi]^d}\frac{d\theta}{c_{1}|\theta|^{\alpha}}.
\]
Here, all the integrals converge when $\alpha<d$ and diverge when $\alpha\geq
d$. If $G_\lambda(x,y) \to \infty$ as $\lambda \to 0$ then $\|\Gamma(\lambda)
\| \to \infty$ and the principal eigenvalue of matrix $\Gamma(\lambda)$ tends
to infinity as $\lambda \to 0$, $\gamma_0(\lambda)\to\infty$. Hence in this
case for all $\beta>0$ the equation $\gamma_0(\lambda)\beta=1$ has a solution
(with respect to $\lambda$) and by definition of $\beta_{c}$ we have that
$\beta_c = 0$.

Let now $G_0 < \infty$, then $G_0(x,y) < \infty$ for all $x$ and $y$. So, in
this case $\|\Gamma(0)\|<\infty$ and, moreover, $\Gamma(\lambda)\to
\Gamma(0)$ as $\lambda\to0$. Then there exists $\gamma_{*}<\infty$ such that
$\gamma_0(\lambda)\le\gamma_{*} < \infty$ for all $\lambda$. In this case the
equation $\gamma_0(\lambda)\beta=1$ does not have solutions (with respect to
$\lambda$) as $\beta \to 0$. By Lemma~\ref{L1} in this case the operator
$\mathscr{H}_\beta$ does not have eigenvalues when $\beta$ is small, i.e.
$\beta_c> 0$.

It remains to prove the upper bound for $\beta_c$. Let $\beta$ be an
arbitrary value of the parameter such that the operator $\mathscr{H}_\beta$
has a positive eigenvalue $\lambda$. Then by Lemma~\ref{L1} we have $
\gamma_{0}(\lambda)\beta=1, $ and hence
\begin{equation}
\label{E:betaest}
\beta=\frac{1}{\gamma_{0}(\lambda)}.
\end{equation}

By Lemma~\ref{L1} the matrix $\Gamma(\lambda)$ is positive. Then the
Perron-Frobenius theorem~\cite[Thm.~8.2.11]{HJ:e} implies that the principal
eigenvalue $\gamma_0(\lambda)$ of the matrix $\Gamma(\lambda)$ has a
corresponding eigenvector $x(\lambda)$ with all positive coordinates.

Let us represent the matrix $\Gamma(\lambda)$ as
\begin{equation*}
\Gamma(\lambda)=G_{\lambda}(0,0)I+ B(\lambda),
\end{equation*}
where $B(\lambda)=\Gamma(\lambda)-G_{\lambda}(0,0)I$. Then, in the case $N\ge
2$, all elements of the matrix $B(\lambda)$ are non-negative while its
off-diagonal elements are positive. Therefore, by definition of the
eigenvector $x(\lambda)$ the following equalities hold:
\begin{equation*}
0=\Gamma(\lambda)x(\lambda)-\gamma_{0}(\lambda)x(\lambda)=
(G_{\lambda}(0,0)-\gamma_{0}(\lambda))x(\lambda)+B(\lambda)x(\lambda).
\end{equation*}
Hence, when $\lambda=0$,
\begin{equation*}
0=\Gamma(0)x(0)-\gamma_{0}(0)x(0)=
(G_{0}(0,0)-\gamma_{0}(0))x(0)+B(0)x(0).
\end{equation*}
The vector $x(0)$ has positive coordinates and therefore the vector
$B(0)x(0)$ also has positive coordinates. So, the last equation holds only if
$ G_{0}(0,0)-\gamma_{0}(0)<0. $ Then by \eqref{E:betaest} we obtain
\begin{equation*}
\beta_{c}=\frac{1}{\gamma_{0}(0)}<\frac{1}{G_{0}(0,0)}=\frac{1}{G_{0}}.
\end{equation*}

For $N=1$ the critical value $\beta_c$ can be found from the equation
$\beta_c G_0 = 1$ and equals $\beta_c = \frac{1}{G_0}$.
\end{proof}

\section{Structure of the Positive Discrete Spectrum}\label{S-WS-BRW}

If there exists $\varepsilon_{0}>0$ such that the operator
$\mathscr{H}_{\beta}$ has a simple positive eigenvalue $\lambda(\beta)$ when
$\beta\in(\beta_{c},\beta_{c}+\varepsilon_{0})$ and this eigenvalue satisfies
a condition $\lambda(\beta)\to 0$ as $\beta\downarrow \beta_{c}$, then we
call supercritical BRW \emph{weakly supercritical} when $\beta$ is close to
$\beta_{c}$ \cite{Y15-DAN:e}.

In connection with this definition, the question naturally arises
\emph{whether every supercritical BRW is weakly supercritical}? The
affirmative answer to this question is given in Theorem~\ref{T-new-2} below
which is a straightforward corollary of the following stronger statement.

\begin{theorem}\label{T-new-3}
Let the  transition intensities $a(z)$ satisfy \eqref{E:findisp} or
\eqref{E:infindisp}, and let $N\geq 2$. Then the operator
$\mathscr{H}_{\beta}$ may have no more than $N$ positive eigenvalues
$\lambda_{i}(\beta)$ of finite multiplicity when $\beta>\beta_{c}$, and
\begin{equation*}
\lambda_{0}(\beta)>\lambda_{1}(\beta)\geq\cdots\geq\lambda_{N-1}(\beta)>0,
\end{equation*}
where the principal eigenvalue $\lambda_{0}(\beta)$ has unit  multiplicity.
Besides there is a value $\beta_{c_{1}}>\beta_{c}$ such that for $\beta\in
(\beta_{c},\beta_{c_{1}})$ the operator $\mathscr{H}_{\beta}$ has a single
positive eigenvalue, $\lambda_{0}(\beta)$.
\end{theorem}

\begin{corollary}
Under the conditions of Theorem~\ref{T-new-3} the multiplicity of each of the
eigenvalues $\lambda_{1}(\beta),\ldots,\lambda_{N-1}(\beta)$ does not exceed
$N-1$.
\end{corollary}

\begin{theorem}\label{T-new-2}
Every supercritical BRW is weakly supercritical as $\beta \downarrow
\beta_c$.
\end{theorem}

Before start proving Theorem~\ref{T-new-3} we establish some further
properties of the matrix $\Gamma(\lambda)$, see definition \eqref{Def:G}.

\begin{lemma}\label{L-Gprop}
For all $\lambda>0$ the elements $\Gamma_{ij}(\lambda)$ of the matrix
$\Gamma(\lambda)$ are continuous, decreasing in $\lambda$ and satisfy
\begin{equation}\label{E-Gineq}
0<\Gamma_{ij}(\lambda)<\Gamma_{11}(\lambda)=\cdots=\Gamma_{NN}(\lambda),\quad i\neq j.
\end{equation}
\end{lemma}

\begin{proof}
The fact that the elements $\Gamma_{ij}(\lambda)=G_{\lambda}(x_{i},x_{j})$
are continuous, positive and decreasing in $\lambda$ immediates from
\eqref{E-Green}. To prove inequalities \eqref{E-Gineq} let us observe that by
\eqref{E-Green} we can write
\begin{equation}\label{E-Gij}
\Gamma_{ij}(\lambda)=G_\lambda(x_{i},x_{j}) =
\frac{1}{(2\pi)^d} \int_{[-\pi,\pi]^d} \frac{\cos(\theta, x_{j}-x_{i})}{\lambda -
\phi(\theta)}d\theta.
\end{equation}
Here, for $i\neq j$, we have $x_{j}\neq x_{i}$ and then $\cos(\theta,
x_{j}-x_{i})<1$ on a subset of $[-\pi,\pi]^d$ of positive measure, which
implies
\[
\Gamma_{ij}(\lambda)=G_\lambda(x_{i},x_{j}) <
\frac{1}{(2\pi)^d} \int_{[-\pi,\pi]^d} \frac{d\theta}{\lambda -
\phi(\theta)}= \Gamma_{11}(\lambda)=\cdots=\Gamma_{NN}(\lambda).
\]
The lemma is proved.
\end{proof}

The next lemma gives a bit more precise characterization of the localization
of the eigenvalues of the matrix $\Gamma(\lambda)$. Let us observe that by
Lemma~\ref{L2} each of the eigenvalues $\gamma_{i}(\lambda)$ of the matrix
$\Gamma(\lambda)$ is strictly decreasing when $\lambda\ge0$. Hence, for each
$i=0,1,\ldots,N-1$ there exists $\lim_{\lambda\to0}\gamma_{i}(\lambda)$,
finite or infinite, which will be denoted as $\gamma_{i}(0)$:
\[
\gamma_{i}(0):=\lim_{\lambda\to0}\gamma_{i}(\lambda),\quad i=0,1,\ldots,N-1.
\]

\begin{lemma}\label{L-eig-localization}
Let the  transition intensities $a(z)$ satisfy \eqref{E:findisp} or
\eqref{E:infindisp}. Then there exists $\gamma^{*}>0$,
$\gamma^{*}<\gamma_{0}$, such that
\begin{equation}\label{E-subperipheral}
0\le\gamma_{N-1}(0)\le\ldots\le \gamma_{1}(0)\le\gamma^{*}<\infty.
\end{equation}
\end{lemma}

\begin{proof}
By Lemma~\ref{L-Gprop} the function $\Gamma_{11}(\lambda)~
(\Gamma_{11}(\lambda)=\Gamma_{22}(\lambda)=\cdots=\Gamma_{NN}(\lambda))$ has
a limit $\lim_{\lambda\to0}\Gamma_{11}(\lambda)$. Our further proof will
depend on whether this limit is finite or infinite.

First, let
\[
\Gamma_{11}(0):=\lim_{\lambda\to0}\Gamma_{11}(\lambda)<\infty.
\]
By Lemma~\ref{L-Gprop} in this case each element $\Gamma_{ij}(\lambda)$ has a
finite positive limit as $\lambda\to0$, $\lambda>0$. This means that there
exists a positive matrix $\Gamma(0)$ such that
\[
\Gamma(\lambda)\to\Gamma(0)\quad\text{as}\quad \lambda\to0,~\lambda>0.
\]
Then by the Perron-Frobenius theorem~\cite[Thm.~8.2.11]{HJ:e}
\[
0\le\gamma_{N-1}(0)\le\ldots\le \gamma_{1}(0)<\gamma_{0}
\]
and we can take $\gamma^{*}=\gamma_{1}(0)$.

Now, let us consider the case
\[
\Gamma_{11}(0):=\lim_{\lambda\to0}\Gamma_{11}(\lambda)=\infty.
\]
Then also
\[
\Gamma_{ii}(0):=\lim_{\lambda\to0}\Gamma_{ii}(\lambda)=\infty,\quad i=1,2,\ldots,N.
\]
In this case we hardly can make use of Perron-Frobenius theorem because the
limiting matrix $\Gamma(0)$ is `infinite'. So, we will behave differently.

Denote by $\boldsymbol{1}$ the $N\times N$ matrix, all elements of which are
$1$'s. Then, by denoting
\[
G_{\lambda}:=\Gamma_{11}(\lambda)=\cdots=\Gamma_{NN}(\lambda)
\]
we can represent the matrix $\Gamma(\lambda)$, for $\lambda>0$, as follows
\begin{equation}\label{E-ggg}
\Gamma(\lambda)=G_{\lambda}\boldsymbol{1}+\tilde{\Gamma}(\lambda),
\end{equation}
where
\[
\tilde{\Gamma}(\lambda)=\left[\begin{array}{cccc}
0&\Gamma_{12}(\lambda)-G_{\lambda}   &\cdots&\Gamma_{1N}(\lambda)-G_{\lambda} \\
\Gamma_{21}(\lambda)-G_{\lambda} &0  &\cdots&\Gamma_{2N}(\lambda)-G_{\lambda} \\
\cdots &\cdots&\cdots\\
\Gamma_{N1}(\lambda)-G_{\lambda}  &\Gamma_{N2}(\lambda)-G_{\lambda}  &\cdots&0 \\
\end{array}\right].
\]
Direct calculations show that the symmetric matrix
$G_{\lambda}\boldsymbol{1}$ has the simple eigenvalue $(N-1)G_{\lambda}$ and
the eigenvalue $0$ of multiplicity $N-1$.

Now, consider the matrix $\tilde{\Gamma}(\lambda)$. For its elements, by
\eqref{E-Gij} we have
\[
\Gamma_{ij}(\lambda)-G_{\lambda}=
\frac{1}{(2\pi)^d} \int_{[-\pi,\pi]^d} \frac{\cos(\theta, x_{j}-x_{i})}{\lambda -
\phi(\theta)}d\theta-
\frac{1}{(2\pi)^d} \int_{[-\pi,\pi]^d} \frac{1}{\lambda -
\phi(\theta)}d\theta.
\]
Then
\begin{align*}
|\Gamma_{ij}(\lambda)-G_{\lambda}|&\le
\frac{1}{(2\pi)^d} \int_{[-\pi,\pi]^d} \frac{|\cos(\theta, x_{j}-x_{i})-1|}{\lambda -
\phi(\theta)}d\theta\\ &\le
\frac{2}{(2\pi)^d} \int_{[-\pi,\pi]^d} \frac{\sin^{2}\frac{(\theta, x_{j}-x_{i})}{2}}{\lambda -
\phi(\theta)}d\theta\le
\frac{|x_{j}-x_{i}|^{2}}{2(2\pi)^d} \int_{[-\pi,\pi]^d}\frac{|\theta|^{2}}{\lambda -
\phi(\theta)}d\theta.
\end{align*}
Therefore,
\begin{equation}\label{E-mainest}
\limsup_{\lambda\to0}|\Gamma_{ij}(\lambda)-G_{\lambda}|\le
\frac{|x_{j}-x_{i}|^{2}}{2(2\pi)^d} \int_{[-\pi,\pi]^d}\frac{|\theta|^{2}}{|\phi(\theta)|}d\theta.
\end{equation}

In the case \eqref{E:findisp}, as was shown in \cite{YarBRW:e}, the function
$\phi(\theta)$ satisfies the estimation $|\phi(\theta)|\ge C|\theta|^{2}$ for
$\theta\in[-\pi,\pi]^d$. In this case condition \eqref{E-mainest} takes the
form
\[
\limsup_{\lambda\to0}|\Gamma_{ij}(\lambda)-G_{\lambda}|\le
\frac{|x_{j}-x_{i}|^{2}}{2C(2\pi)^d} \int_{[-\pi,\pi]^d}\frac{|\theta|^{2}}{|\theta|^{2}}d\theta=
\frac{|x_{j}-x_{i}|^{2}}{2C}<\infty.
\]

In the case \eqref{E:infindisp} due to \eqref{E-alphaest} the function
$\phi(\theta)$ satisfies the estimation $|\phi(\theta)|\ge
C|\theta|^{\alpha}$ for $\theta\in[-\pi,\pi]^d$, where $\alpha\in (0,2)$. In
this case condition \eqref{E-mainest} takes the form
\[
\limsup_{\lambda\to0}|\Gamma_{ij}(\lambda)-G_{\lambda}|\le
\frac{|x_{j}-x_{i}|^{2}}{2C(2\pi)^d} \int_{[-\pi,\pi]^d}\frac{|\theta|^{2}}{|\theta|^{\alpha}}d\theta,
\]
where the integral in the right-hand part is converging for all $\alpha$
which satisfy $\alpha-2<d$, and hence for all $\alpha\in (0,2)$.

So, under conditions \eqref{E:findisp} and \eqref{E:infindisp} we have the
estimate
\begin{equation}\label{E:limest}
\limsup_{\lambda\to0}|\Gamma_{ij}(\lambda)-G_{\lambda}|\le C^{*},\quad i,j=1,2,\ldots,N,
\end{equation}
for all $\lambda>0$ sufficiently close to zero, where
\[
C^{*}=\max_{i,j}\frac{|x_{j}-x_{i}|^{2}}{2C(2\pi)^d} \int_{[-\pi,\pi]^d}\frac{|\theta|^{2}}{|\theta|^{\alpha}}d\theta,\quad \alpha\in(0,2].
\]
Hence, for such values of $\lambda>0$ the norm of the matrix
$\tilde{\Gamma}(\lambda)$ is uniformly bounded by some constant, and
consequently its spectral radius $\rho(\tilde{\Gamma}(\lambda))$ (i.e. the
maximal absolute value of its eigenvalues) is also uniformly bounded by some
constant $\gamma^{*}$, that is
\begin{equation}\label{E-gammastar}
\rho(\tilde{\Gamma}(\lambda))\le\gamma^{*}
\end{equation}
for all $\lambda>0$ sufficiently close to zero.

Now, observe that all matrices in \eqref{E-ggg} are symmetric and then by the
Weyl theorem \cite[Thm.~4.3.1]{HJ:e} the eigenvalues of the matrix
$\Gamma(\lambda)$ differ from the corresponding eigenvalues of the matrix
$G_{\lambda}\boldsymbol{1}$ by no more than $\gamma^{*}$, that is
\begin{align}\label{E-gl0}
|\gamma_{0}(\lambda)-(N-1)G_{\lambda}|&\le \gamma^{*},\\
\label{E-gl1}
0\le\gamma_{N-1}(\lambda)\le\ldots\le \gamma_{1}(\lambda)&\le\gamma^{*}.
\end{align}

Passing to the limit in \eqref{E-gl1}, we obtain the required estimate
$\gamma_{1}(0)\le\gamma^{*}$. The lemma is proved.
\end{proof}

\begin{remark}\rm
For the case \eqref{E:findisp} of finite variance of jumps the statement of
Lemma~\ref{L-eig-localization} was presented in \cite{AY15-MPMM:e,Y15-DAN:e}.
Here we give not only the proof for this case but also for the case
\eqref{E:infindisp} of infinite variance of jumps.
\end{remark}

\begin{remark}\rm
Sometimes it might be useful to know explicit estimates for the `spectral
gap' $\gamma^{*}$.

As was shown in the proof of Lemma~\ref{L-eig-localization}, in the case
$\lim_{\lambda\to0}\Gamma_{11}(\lambda)<\infty$ the quantity $\gamma^{*}$ can
be estimate as follows: $\gamma^{*}:=\gamma_{1}(0)$. Define
\[
M:=\max_{i,j} \Gamma_{ij}(0),\quad m:=\min_{i,j} \Gamma_{ij}(0),
\]
then by Lemma~\ref{L-Gprop} $m<M$, and the Hopf theorem
\cite[Sect.~8.2.12]{HJ:e} implies the following estimate for the so-called
`spectral gap':
\[
\gamma^{*}=\gamma_{1}(0)\le\frac{M-m}{M+m}\gamma_{0}(0)<\gamma_{0}(0).
\]

In the case $\lim_{\lambda\to0}\Gamma_{11}(\lambda)=\infty$, under conditions
\eqref{E:findisp} or \eqref{E:infindisp}, due to \eqref{E-gammastar} the
quantity $\gamma^{*}$ can be defined as
$\gamma^{*}:=\limsup_{\lambda\to0}\rho(\tilde{\Gamma}(\lambda))$. But since
for the elements of the matrix $\tilde{\Gamma}(\lambda)$ the limiting
estimate \eqref{E:limest} holds, then by \cite[Corollary~6.1.5]{HJ:e}
$\limsup_{\lambda\to0}\rho(\tilde{\Gamma}(\lambda))\le (N-1)C^{*}$ and thus
$\gamma^{*}\le (N-1)C^{*}$.
\end{remark}

\begin{proof}[Proof of Theorem~\ref{T-new-3}]
By Lemma~\ref{L1} the eigenvalues of the operator $\mathscr{H}_\beta$ satisfy
the equations \eqref{E:glbeq1}. The quantities $\gamma_{i}(\lambda)$ are the
eigenvalues of the positive and symmetric matrix $\Gamma(\lambda)$, and (as
is shown in \cite{MolYar12-2:e} and also follows from Lemma~\ref{L2}) this
matrix is the Gramian matrix for some appropriate scalar product and hence is
positive-definite. In this case by Lemma~\ref{L-Gprop} all eigenvalues
$\gamma_{i}(\lambda)$ are real and positive (and can be arranged in ascending
order):
\begin{equation}\label{E-eigen}
0\le\gamma_{N-1}(\lambda)\le\ldots\le \gamma_{1}(\lambda)\le \gamma_{0}(\lambda).
\end{equation}
From \eqref{E-Green} it follows that elements of the matrix $\Gamma(\lambda)$
tend to zero as $\lambda \to \infty$ and so $\gamma_i(\lambda) \to 0$ as
$\lambda \to \infty$ for all $i=0,1,\ldots,N-1$. By Rellich theorem
\cite[Ch.~2, Thm.~6.8]{Kato:e} all functions $\gamma_i(\lambda)$ are
piecewise smooth when $\lambda\ge0$.

By Lemma~\ref{L-Gprop} the elements of the matrix $\Gamma(\lambda)$ are
strictly positive when $\lambda>0$ and then by Perron-Frobenius
theorem~\cite[Thm.~8.2.11]{HJ:e} the principal eigenvalue
$\gamma_{0}(\lambda)$ of the matrix $\Gamma(\lambda)$ has the unit
multiplicity and strictly exceeds other eigenvalues, i.e. the last of the
inequalities \eqref{E-eigen} is strict:
\begin{equation}\label{E-eigen1}
0\le\gamma_{N-1}(\lambda)\le\ldots\le \gamma_{1}(\lambda)< \gamma_{0}(\lambda).
\end{equation}

Again by Lemma~\ref{L-Gprop} the matrix $\Gamma(\lambda)$ is  continuous for
all values $\lambda>0$. Behaviour of this matrix can differ as $\lambda\to 0$
and further proof of the theorem depends on this behaviour.

As we noted earlier, the quantity
$\Gamma_{11}(\lambda)=\cdots=\Gamma_{NN}(\lambda)=G_{\lambda}(0,0)$ has a
limit as $\lambda\to0$, finite or infinite. If the quantity
$G_{\lambda}(0,0)$ tends to some finite limit as $\lambda\to 0$ then by
Lemma~\ref{L-eig-localization} the eigenvalues $\gamma_i(\lambda)$ behave as
is shown in Fig.~\ref{F-mulambda}, where
$\gamma_{1}(0)=\frac{1}{\beta_{c_{1}}}\le \gamma^{*}$. In the case when the
transition intensities $a(z)$ satisfy \eqref{E:findisp} this is true for
$d\ge 3$. In the case when the transition intensities $a(z)$ satisfy
\eqref{E:infindisp} this is true for $d=1$ and $\alpha\in(0,1)$, or for $d\ge
2$ and $\alpha\in(0,2)$.

\begin{figure}[!htbp]
\hfill\subfigure[Case $\lim\limits_{\lambda\to\infty}\gamma_{0}(\lambda)<\infty$.]
{\includegraphics[width=0.4\textwidth]{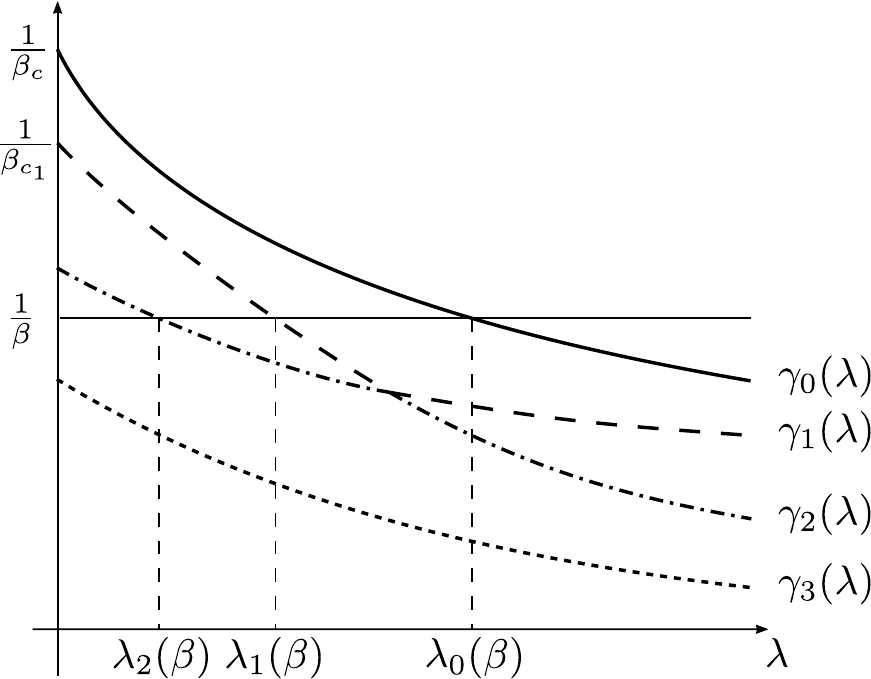}}
\hfill\subfigure[Case
$\lim\limits_{\lambda\to\infty}\gamma_{0}(\lambda) =\infty$.]
{\includegraphics[width=0.4\textwidth]{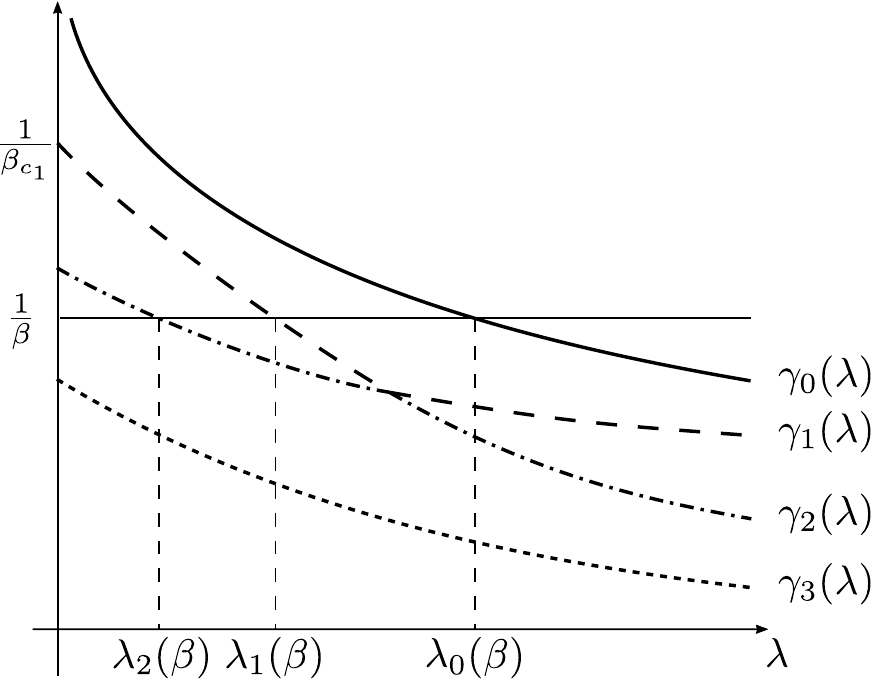}}
\hfill\mbox{}
\caption{Plots of the functions $\gamma_{i}(\lambda)$.}\label{F-mulambda}
\end{figure}

Another possible case is when $G_{\lambda}(0,0)\to\infty$ as $\lambda\to 0$.
For positive matrices, the leading eigenvalue always exceeds any diagonal
element of the matrix. Therefore, in this case
$\gamma_{0}(\lambda)\ge\Gamma_{11}(\lambda)=G_{\lambda}(0,0)$, from which
$\gamma_{0}(\lambda) \to\infty$ as $\lambda\to 0$. Then $\beta_{c}=0$. If the
transition intensities $a(z)$ satisfy \eqref{E:findisp} this is true for
$d=1$ or $d=2$. If the transition intensities $a(z)$ satisfy
\eqref{E:infindisp} this is true for $d=1$ and $\alpha\in[1,2)$. By
Lemma~\ref{L-eig-localization} here again
$\beta_{c_{1}}\ge\frac{1}{\gamma^{*}}>0$. This situation is illustrated in
Fig.~\ref{F-mulambda}(b).


By Lemma~\ref{L1} the eigenvalues $\lambda_i$ of the operator
$\mathscr{H}_\beta$ are solutions of the equations \eqref{E:glbeq1}, see
Fig.~\ref{F-mulambda}. However, since by Lemma~\ref{L2} every function
$\gamma_{i}(\lambda)$ is strictly decreasing when $\lambda\ge0$, then the
total number of the solutions $\lambda_{i}(\beta)$ of these equations does
not exceed $N$.

From inequalities \eqref{E-eigen1} it follows that if the operator
$\mathscr{H}_\beta$ has positive eigenvalues for some fixed $\beta$, then the
maximal of them is $\lambda_{0}=\lambda_{0}(\beta)$ which is a solution of
the equation
\begin{equation}\label{E-lambda0}
\gamma_0(\lambda_{0}) = \frac{1}{\beta},
\end{equation}
is simple and strictly exceeds others. The minimum value of the solution
$\lambda_0$ of equation~\eqref{E-lambda0} is $\lambda_0 = 0$. The
corresponding value $\beta_c$ of the parameter $\beta$ is critical. Since the
function $\gamma_{0}(\lambda)$ is strictly decreasing then the eigenvalue
$\lambda_{0}=\lambda_{0}(\beta)$ increases with the increase of parameter
$\beta$. The theorem is proved.
\end{proof}

\section{Multiple Eigenvalues: an Example}\label{S-example}

By Theorem~\ref{T-new-3} the principal eigenvalue $\lambda_{0}(\beta)$ of the
operator $\mathscr{H}_{\beta}$ is always simple. In this section we
demonstrate that some other eigenvalues $\lambda_{1}(\beta), \ldots,
\lambda_{N-1}(\beta)$ of this operator may actually coincide (i.e. their
multiplicity may be greater than one) and this situation is possible even in
the case of arbitrary finite number of sources (of equal intensity). As is
shown in the following Example~\ref{Ex-new-4} this situation may occur if
there is a certain `symmetry' of the spatial configuration of the sources
$x_{1}, x_{2}, \ldots, x_{N}$.

In Example~\ref{Ex-new-4} we assume that the function of transition
probabilities is symmetric in the following sense: its values do not change
at any permutation of arguments. In particular, a function of a vector
variable $z$ is symmetric if its values are the same at any permutation of
coordinates of vector $z$.

Let us present a statement to be used further for `integrable' models, where
equations for estimating $\lambda_{i}$ from Theorem~\ref{T-new-3} can be
found explicitly.

\begin{lemma}\label{L-eqviv-G}
If the function of transition probabilities $a(z)$ is symmetric in the
following sense: its values do not change at any permutation of arguments,
then the function $G_{\lambda}(z)$ is also symmetric in the same sence.
\end{lemma}

\begin{proof}
Let us recall that Green's function $G_\lambda(z)$ can be represented
\cite{YarBRW:e} in the form \eqref{E-Green}, where $\phi(\theta)$ is the
Fourier transform of the function of transition probabilities $a(z)$ and is
defined by $\phi(\theta) = \sum_{z\in\mathbb Z^d} a(z) e^{i(\theta,z)}$,
$\theta\in [-\pi,\pi]^d$. To prove the lemma it suffices to demonstrate that
$G_{\lambda}(z)=G_{\lambda}(\mathbf{R}z)$ for every permutation matrix
$\mathbf{R}$ (i.e. a matrix whose rows and columns have the only one nonzero
element, which is equal to one). So,
\begin{align*}
\phi(\mathbf{R}\theta) &= \sum_{z\in\mathbb Z^d} a(z) e^{i(\mathbf{R}\theta,z)} =
\sum_{z\in\mathbb Z^d} a(z) e^{i(\theta, \mathbf{R}^{*}z)}\\&=
\sum_{z'\in\mathbb Z^d} a((\mathbf{R}^{*})^{-1}z') e^{i(\theta, z')}  =
\sum_{z'\in\mathbb Z^d} a(z') e^{i(\theta,z')} =
\phi(\theta),
\end{align*}
where the equality $a((\mathbf{R}^{*})^{-1}z')=a(z')$ holds for all
$z'\in\mathbb Z^d$ since the function $a(z)$ is symmetric. Hence the function
$\phi(\theta)$ is also symmetric. Further,
\begin{align*}
G_\lambda(\mathbf{R}z) &=
\frac{1}{(2\pi)^d} \int_{[-\pi,\pi]^d} \frac{e^{i(\theta,\mathbf{R}z)}}{\lambda-\phi(\theta)} d\theta =
\frac{1}{(2\pi)^d} \int_{[-\pi,\pi]^d} \frac{e^{i(\mathbf{R}^{*}\theta,z)}}{\lambda-\phi(\theta)} d\theta \\&=
\frac{1}{(2\pi)^d} \int_{[-\pi,\pi]^d} \frac{e^{i(\theta,z)}}{\lambda-\phi((\mathbf{R}^{*})^{-1}\theta)} d\theta =
G_\lambda(z)
\end{align*}
for every permutation matrix $\mathbf{R}$, where the equality
$\phi(\theta)=\phi((\mathbf{R}^{*})^{-1}\theta)$ again holds for all $\theta
\in[-\pi,\pi]^d$ since the function $\phi(\theta)$ is symmetric. Thus, the
function $G_{\lambda}(z)$ is also symmetric.
\end{proof}

Now, we are ready to present an example.

\begin{example}\label{Ex-new-4}
Let the function of transition probabilities $a(z)$ be symmetric in the
following sense: its values do not change at any permutation of arguments,
and let $x_{1},\ldots,x_{N}$, where $N\geq 2$, be the vertices of a regular
simplex (i.e. of a simplex for which the lengths of the edges are equal). For
example, we can take
\begin{equation*}
x_{1}=\{1,0,\ldots,0\},~
x_{2}=\{0,1,\ldots0\},~\ldots,~
x_{N}=\{0,0,\ldots,1\}.
\end{equation*}

By Lemma~\ref{L1} the existence of a non-trivial solution $\lambda$ of the
equation~\eqref{E:glbeq1} for some $\beta$ is equivalent to the resolvability
of the equation
 \begin{equation}\label{E-det}
\det\left(\Gamma(\lambda)-\tfrac{1}{\beta}I\right)=0,
\end{equation}
where the matrix $\Gamma(\lambda)$ with the elements
$\Gamma_{ij}(\lambda)=G_{\lambda}(x_{i},x_{j})$ is defined by
equation~\eqref{Def:G}.

Since the random walk by assumption is symmetric and homogeneous then
\[
G_{\lambda}(x_{i},x_{j})=G_{\lambda}(0,x_{i}-x_{j})=G_{\lambda}(0,x_{j}-x_{i})=G_{\lambda}(x_{j}-x_{i}).
\]
From the definition of the function $G_{\lambda}(u,v)\equiv G_{\lambda}(u-v)$
by Lemma~\ref{L-eqviv-G} it follows that all the values
$G_{\lambda}(x_{j}-x_{i})$ coincide with each other when $i\neq j$, and hence
they coincide with $G_{\lambda}(x_{1}-x_{2})=G_{\lambda}(z_{*})$ (for
simplicity we denote $z_{*}=x_{1}-x_{2}$). So,
\begin{equation}
\label{E-defz}
G_{\lambda}(x_{i},x_{j})=G_{\lambda}(x_{j}-x_{i})\equiv G_{\lambda}(x_{1}-x_{2})=G_{\lambda}(z_{*})\quad\text{for all}\quad
i\neq j,
\end{equation}
while for $i=j$ we have $G_{\lambda}(x_{i},x_{i})\equiv
G_{\lambda}(x_{i}-x_{i})=G_{\lambda}(0)=G_{\lambda}$. Thus, we can represent
equation~\eqref{E-det} as
\begin{equation*}
\det\left[\begin{array}{ccc}
G_{\lambda}-\frac{1}{\beta}   &\cdots&G_{\lambda}(z_{*})   \\
G_{\lambda}(z_{*})  &\cdots&G_{\lambda}(z_{*})  \\
\cdots &\cdots&\cdots\\
G_{\lambda}(z_{*})  &\cdots  &G_{\lambda}-\frac{1}{\beta}\\
\end{array}\right]=0.
\end{equation*}
Using a bit cumbersome but standard linear transforms we rewrite the last
determinantal equation as
\begin{equation*}
\left(G_{\lambda}-G_{\lambda}(z_{*})-\frac{1}{\beta}\right)^{N-1}\det\left[\begin{array}{ccc}
G_{\lambda}-\frac{1}{\beta}+(N-1)G_{\lambda}(z_{*})   &\cdots&G_{\lambda}(z_{*})   \\
0  &\cdots& 0\\
\cdots &\cdots&\cdots\\
0  &\cdots  &-1\\
\end{array}\right]=0,
\end{equation*}
which is equivalent to
\begin{equation*}
\left(G_{\lambda}+(N-1)G_{\lambda}(z_{*})-\frac{1}{\beta}\right)
\left(G_{\lambda}-G_{\lambda}(z_{*})-\frac{1}{\beta}\right)^{N-1}=0.
\end{equation*}
From this last equation, it is seen that the operator $\mathscr{H}_{\beta}$
has a simple leading eigenvalue $\lambda=\lambda_{0}(\beta)$ satisfying the
equation
\[
G_{\lambda}+(N-1)G_{\lambda}(z_{*})-\frac{1}{\beta}=0,
\]
and an eigenvalue $\lambda=\lambda_{1}(\beta)=\cdots= \lambda_{N-1}(\beta)$
of multiplicity $N-1$ satisfying the equation
\[
G_{\lambda}-G_{\lambda}(z_{*})-\frac{1}{\beta}=0.
\]
Hence in this case the quantities $\beta_{c}$ and $\beta_{c_{1}}$ can be
calculated explicitly:
\begin{equation}
\label{E-beta}
\beta_{c}=\frac{1}{G_{0}+(N-1)G_{0}(z_{*})}, \quad
\beta_{c_{1}}=\frac{1}{G_{0}-G_{0}(z_{*})}.
\end{equation}
\end{example}

\begin{remark}\label{R-new-1}\rm
Under the conditions of Example~\ref{Ex-new-4} according to \eqref{E-defz}
and \eqref{E-beta} the quantity $\beta_{c_{1}}$ depends on the norm $|z_{*}|$
of the vector $z_{*}$ (i.e. on the distance between the sources) and does not
depend on the number $N$ of the sources, that is
$\beta_{c_{1}}=\beta_{c}(|z_{*}|)>0$. At the same time the quantity
$\beta_{c}$ depends not only on the distance between the sources but also on
the number $N$ of the sources, that is $\beta_{c}=\beta_{c}(|z_{*}|,N)$, and
in such a way that $\beta_{c}(|z_{*}|,N)\to 0$ as $N\to\infty$ when $z_{*}$
is fixed. Moreover, $\beta_{c}(|z_{*}|,N)\equiv0$ when $G_{0}=\infty$.
\end{remark}

\begin{remark}For the case when $\mathscr{A} = \kappa\Delta$, $\kappa > 0$, is the
lattice Laplacian, Example~\ref{Ex-new-4} was presented in \cite{Y15-DAN:e},
and under assumption of finite variance of jumps it was given in
\cite{AY15-MPMM:e}.
\end{remark}

\section*{Acknowledgments}

This study has been carried out at Lomonosov Moscow State University and at
Steklov Mathematical Institute of Russian Academy of Sciences. The work was
supported by the Russian Science Foundation, project no.~14-21-00162.


 \providecommand{\bbljan}[0]{January}
 \providecommand{\bblfeb}[0]{February}
 \providecommand{\bblmar}[0]{March}
 \providecommand{\bblapr}[0]{April}
 \providecommand{\bblmay}[0]{May}
 \providecommand{\bbljun}[0]{June}
 \providecommand{\bbljul}[0]{July}
 \providecommand{\bblaug}[0]{August}
 \providecommand{\bblsep}[0]{September}
 \providecommand{\bbloct}[0]{October}
 \providecommand{\bblnov}[0]{November}
 \providecommand{\bbldec}[0]{December}
 \providecommand{\cprime}[0]{$'$}

\end{document}